\documentclass[12pt,twoside]{amsart}
\usepackage{amsmath}
\usepackage{amsthm}
\usepackage{amsfonts}
\usepackage{amssymb}
\usepackage{latexsym}
\usepackage{mathrsfs}
\usepackage{amsmath}
\usepackage{amsthm}
\usepackage{amsfonts}
\usepackage{amssymb}
\usepackage{latexsym}
\usepackage{geometry}
\usepackage{dsfont}
\usepackage[dvips]{graphicx}
\usepackage{color}
\usepackage[all]{xy}

\date{}
\allowdisplaybreaks[4] \footskip=15pt
\renewcommand{\uppercasenonmath}[1]{}

\usepackage{graphicx,amssymb}
\usepackage[all]{xy}
\usepackage{amsmath}

\allowdisplaybreaks
\usepackage{amsthm}
\usepackage{color}

\theoremstyle{plain}
\newtheorem{theorem}{Theorem}[section]
\newtheorem{proposition}[theorem]{Proposition}
\newtheorem{lemma}[theorem]{Lemma}
\newtheorem{corollary}[theorem]{Corollary}

\newtheorem*{open question}{Open Question}
\newtheorem{definition}[theorem]{Definition}

\theoremstyle{definition}

\theoremstyle{remark}

\newcommand{\Id}{\mathrm{Id}}

\def\Hom{{\rm Hom}}

\def\Ker{{\rm Ker}}

\def\Im{{\rm Im}}

\begin{document}
\begin{center}
{\large  \bf Some remarks on  injective envelopes on ring extensions}

\vspace{0.5cm}   Xiaolei Zhang$^{a}$

{\footnotesize
	School of Mathematics and Statistics, Shandong University of Technology,
	Zibo 255049, China\\
	
	E-mail: zxlrghj@163.com\\}

\end{center}

\bigskip
\centerline { \bf  Abstract}
\bigskip
\leftskip10truemm \rightskip10truemm \noindent

Let $f:S\rightarrow R$ be a ring extension.  We introduce and study the properties of  $(R, S)_\star$-injective modules and the existences of $(R, S)_\star$-injective envelopes. Besides, we show that  every $R$-module has an $(R, S)$-injective envelope when $S$ is a pure-semisimple ring, partly resolving a conjecture proposed by Guo \cite{G24}.
\vbox to 0.3cm{}

\leftskip0truemm \rightskip0truemm
\bigskip

\section{Introduction}
Throughout this paper, $f:S\rightarrow R$ always denotes a ring extension, and all modules are left modules unless otherwise stated. Let $M$ be an $R$-module. We denote by $\Id_M$ the identical automorphism of $M$.

Early in 1956, Hochschild \cite{H56} built the fundamental theory of relative homological algebra on a ring extension. In particular, he introduced the notions of relative projective modules and relative injective modules. Early researchers paid more attention on the relative homological algebras on semisimple extensions, separable extensions and Frobenius extensions \cite{H59, HS66}. In the past three decades,  some algebraists pointed out that there are some connections of relative global dimensions of Artin algebras and the finitistic dimension conjecture \cite{AS93,AS932,AS933,F91,G19,XX13}.

Recently, a series of studies have been conducted by Guo et al. on the classes of modules arised in relative homology algebras, including relative projective (free, flat injective) modules   \cite{G21, G22, G24}. In particular, Guo \cite{G24} defined the  relative injective envelopes of modules and  verified the  existences of relative injective envelopes of finitely generated module over an extension of Artin algebras. Subsequently, he left the following conjecture:

\textbf{Conjecture} \cite{G24}
Let $f :S\rightarrow  R$ be a ring extension. Then every $R$-module has a relative injective envelope.

In this paper, we first obtain that the conjecture is true for all ring extensions from a pure-semisimple ring. And then we introduce and  study introduce a new version of injective modules on a ring extension $f:S\rightarrow R$ by using purity over the ring $S$. In particular, we studies their enveloping properties in the last section.

\section{Relative injective envelopes on ring extensions from pure-semisimple rings}

Let  $f:S \rightarrow R$ be a ring extension. Then every $R$-module can be naturally regarded as an $S$-module.
Recall from \cite{H56} that a short exact sequence $0\rightarrow  U\xrightarrow{f}
V\xrightarrow{g}  W\rightarrow  0$ of $R$-modules is called \emph{$(R, S)$-exact} provided that it splits  as  $S$-modules.
An $R$-module $M$ is said to be \emph{$(R, S)$-injective}, if for every $(R, S)r$-exact sequence $0\rightarrow  U\xrightarrow{u}
V\xrightarrow{v}  W\rightarrow  0$
and every $R$-homomorphism $t :U\rightarrow M$, there is an $R$-homomorphism $s :V \rightarrow M$
such that $t=su$, that is, the following commutative diagram with the bottom row $(R, S)$-exact can be completed:
$$\xymatrix@R=30pt@C=50pt{	&& M&&\\
	0\ar[r]&U\ar[r]_{u}\ar[ru]^{t} &V\ar@{.>}[u]_{s}\ar[r]_v&W\ar[r]&0 \\
}$$
Similarly, the notions of    \emph{$(R, S)$-projective} (resp.,  \emph{$(R, S)$-flat})  modules are defined to be $R$-modules which are projective  (resp.,  flat) with respect to  all short $(R, S)$-exact sequences (see \cite{G22}). It was proved in \cite{H56} that $\Hom_S(R,M)$ (resp., $M\otimes_SR$) is always $(R, S)$-injective (resp., $(R, S)$-projective) for any left (resp., right) $S$-module $M$.

Recently, Guo \cite{G24} defined
\emph{$(R, S)$-injetive envelope}  of an $R$-module $M$: An $(R, S)$-monomorphism  $\sigma:M\rightarrow E$ with $E$ $(R,S)$-injective is an $(R, S)$-injetive envelope if it is left minimal, i.e., every endomorphism $h$ of $E$ such that $\sigma=h\sigma$  is an automorphism.

The author in \cite{G24} pointed out that suppose $f:S\rightarrow R$ is an extension of Artin algebras.
Then every finitely generated $R$-module has an  $(R, S)$-injective envelope. Subsequently, he proposed the following conjecture:

\textbf{Conjecture} \cite{G24}
Let $f :S\rightarrow  R$ be a ring extension. Then every $R$-module has an $(R, S)$-injective envelope.\\
We will show the conjecture is true for any ring extension from a pure-semisimple ring. First we recall the classical notions of envelopes and preenvelopes.

Let $M$ be an $R$-module and $\mathscr{C}$ be a  class  of $R$-modules. Recall from \cite{EJ11} that an $R$-homomorphism $f:M\rightarrow C$ with  $C\in \mathscr{C}$ is said to be a  \emph{$\mathscr{C}$-preenvelope}  provided that  for any $C'\in \mathscr{C}$, the natural homomorphism  $\Hom_{R}(C,C')\rightarrow \Hom_{R}(M,C')$ is an epimorphism, that is,  for any $R$-homomorphism $f':M\rightarrow C'$ with $C'\in \mathscr{C}$, there is an $R$-homomorphism $h:C\rightarrow C'$ such that the following diagram is commutative:$$\xymatrix@R=20pt@C=50pt{M\ar[r]^{f}\ar[rd]_{f'} &C\ar@{.>}[d]^{h} \\
	& C'\\}$$
If, moreover, every endomorphism $h$ such that $f=hf$  is an automorphism, that is, the commutativity of
$$\xymatrix@R=20pt@C=50pt{M\ar[r]^{f}\ar[rd]_{f} &C\ar@{.>}[d]^{h} \\
	& C\\}$$ implies $h$ is an automorphism,  then $f$ is said to be a  \emph{$\mathscr{C}$-envelope}.

Let $M$ be an $R$-module. We always write by  $M^{+}:=\Hom_\mathbb{Z}(M,\mathbb{Q}/\mathbb{Z})$. We always set $\epsilon:M\rightarrow  \Hom_S(R,M)$ to be the $R$-homomorphism satisfying $\epsilon(m)(r)=rm$ for any $m\in M$ and $r\in R$. Then $\epsilon$ is an $(R, S)$-monomorphism.
\begin{proposition}\label{equal-c}
	Let $\mathscr{E}_{(R,S)}$ be the  class  of all $(R,S)$-injective modules. Suppose $\sigma:M\rightarrow E$ is an $R$-homomorphism with $E\in \mathscr{E}_{(R,S)}$. Then $\sigma$ is an  $\mathscr{E}_{(R,S)}$-envelope if and only if it is an $(R,S)$-injective envelope.
\end{proposition}
\begin{proof} Let $\sigma:M\rightarrow E$ be an $\mathscr{E}_{(R,S)}$-envelope of $M$. Then $\sigma$ is an $(R,S)$-monomorphism. Indeed, let $\epsilon:M\rightarrow\Hom_S(R,M)$ be the natural $R$-homomorphism. Since $\Hom_S(R,M)$ is  $(R,S)$-injective, then there is an $R$-homomorphism $h:E\rightarrow\Hom_S(R,M)$ such that the following diagram  commutative
	$$\xymatrix@R=20pt@C=50pt{M\ar[r]^{\sigma}\ar@{_{(}->}[rd]_{\epsilon} &E\ar@{.>}[d]^{h} \\
		& \Hom_S(R,M)\\}$$
	Since $\epsilon$ is an 	$(R,S)$-monomorphism, so is  $\sigma$. Since $\sigma$ is left minimal and $E$ is $(R,S)$-injective, $\sigma$ is also an  $(R,S)$-injective envelope.
	
	On the other hand, suppose $\sigma:M\hookrightarrow E$ is an  $(R,S)$-injective envelope. Then for any   $R$-homomorphism $f':M\rightarrow E'$ with $E'\in \mathscr{E}_{(R,S)}$ there is an $R$-homomorphism $h:E\rightarrow E'$ such that the following diagram is commutative:$$\xymatrix@R=20pt@C=50pt{M\ar@{^{(}->}[r]^{\sigma}\ar[rd]_{f'} &E\ar@{.>}[d]^{h} \\
		& E'\\}$$
	Note that $\sigma$ is left minimal. Consequently, $\sigma$ is an $\mathscr{E}_{(R,S)}$-envelope of $M$.	
\end{proof}

To study the enveloping properties of $(R,S)$-injective modules,  we  recall the notion of injective structure on $R$-modules (see \cite{EJ11}). Let $\mathscr{M}$ be a class of $R$-homomorphisms  and $\mathscr{F}$ be a class of $R$-modules. A pair $(\mathscr{M},\mathscr{F})$ is called an \emph{injective structure} on $R$-modules provided that
\begin{enumerate}
	\item  $F\in \mathscr{F}$ if and only if $\Hom_R(N,F)\rightarrow \Hom_R(M,F)\rightarrow 0$ is exact for any $M\rightarrow N\in \mathscr{M}$
	\item $M\rightarrow N\in \mathscr{M}$ if and only if $\Hom_R(N,F)\rightarrow \Hom_R(M,F)\rightarrow 0$ is exact for any $F\in \mathscr{F}$.
	\item $\mathscr{F}$ is preenveloping.
\end{enumerate}

Let $\mathscr{P}$ be a class of right $R$-modules. We say that $(\mathscr{M},\mathscr{F})$ is \emph{determined by $\mathscr{P}$} provided that the following condition holds:
\begin{center}
	$M\rightarrow N\in \mathscr{M}$ if and only if $0\rightarrow G\otimes_RM\rightarrow G\otimes_RN$ is exact for any $G\in \mathscr{P}$.
\end{center}
It follows from \cite[Theorem 6.6.4]{EJ11} that if an injective structure $(\mathscr{M},\mathscr{F})$ is determined by a class $\mathscr{P}$, then $\mathscr{F}$ is enveloping.

\begin{proposition}\label{injstru}
	Let $\mathscr{E}_{(R,S)}$ be  the  class  of all $(R,S)$-injective modules and $\mathscr{A}_{(R,S)}$  the  class  of all $(R,S)$-monomorphisms. Then $(\mathscr{A}_{(R,S)},\mathscr{E}_{(R,S)})$ is an injective structure on $R$-modules.
\end{proposition}
\begin{proof}  Trivially, $F\in \mathscr{E}_{(R,S)}$ if and only if $\Hom_R(N,F)\rightarrow \Hom_R(M,F)\rightarrow 0$ is exact for any $M\rightarrow N\in \mathscr{A}_{(R,S)}$; and if $M\rightarrow N\in \mathscr{A}_{(R,S)}$, then  $\Hom_R(N,F)\rightarrow \Hom_R(M,F)\rightarrow 0$ is exact for any $F\in \mathscr{E}_{(R,S)}$.
	
	Now let $\sigma:M\rightarrow N$ be an $R$-homomorphism such that $\Hom_R(N,F)\rightarrow \Hom_R(M,F)\rightarrow 0$ is exact for any $F\in \mathscr{E}_{(R,S)}$. Let $\epsilon:M\rightarrow\Hom_S(R,M)$ be the natural $R$-homomorphism. Since $\Hom_S(R,M)\in \mathscr{E}_{(R,S)}$,  there is an $R$-homomorphism $h:N\rightarrow \Hom_S(R,M)$ such that the following diagram is commutative:
	$$\xymatrix@R=40pt@C=50pt{M\ar[r]^{\sigma}\ar@{_{(}->}[rd]_{\epsilon} &N\ar@{.>}[d]^{h} \\
		& \ar@/_1pc/[lu]_{\epsilon'}\Hom_S(R,M)\\}$$
Let $\epsilon':\Hom_S(R,M)\rightarrow\Hom_S(S,M)\cong M$	be the natural $S$-homomorphism. Then $\epsilon'\epsilon=\Id_M$. Hence $\epsilon'h\sigma=\epsilon'\epsilon=\Id_M.$
	Consequently, $\sigma\in \mathscr{A}_{(R,S)}$.
	
	Let $M$ be an $R$-module. Then the natural $R$-homomorphism $\epsilon:M\rightarrow\Hom_S(R,M)$  is clearly an $\mathscr{E}_{(R,S)}$-preenvelope of $M$. Consequently, $(\mathscr{A}_{(R,S)},\mathscr{E}_{(R,S)})$ is an injective structure on $R$-modules.
\end{proof}

\begin{theorem}\label{psprsijn}
	Let $S$ be a pure-semisimple ring and $f:S\rightarrow R$ be a ring extension. Then injective structure $(\mathscr{A}_{(R,S)},\mathscr{E}_{(R,S)})$ is determined by $$\mathscr{F}=\{F\otimes_SR\mid F\ \mbox{is a finitely presented right}\ S\mbox{-module}\}.$$ Consequently, every $R$-module has an $(R, S)$-injective envelope.
\end{theorem}
\begin{proof} Set $$\mathscr{F}=\{F\otimes_SR\mid F\ \mbox{is a finitely presented right}\ S\mbox{-module}\}.$$
	We will show $(\mathscr{A}_{(R,S)},\mathscr{E}_{(R,S)})$ is determined by $\mathscr{F}$.
	It follows by \cite[Lemma 2]{H56} that $R\otimes_SF$ is $(R,S)$-projective, and thus $(R,S)$-flat by \cite[Proposition 3.2]{G22}. Hence if $M\rightarrow N$ is an $(R,S)$-monomorphism, then $0\rightarrow G\otimes_RM\rightarrow G\otimes_RN$ is exact for any $G\in \mathscr{F}$. On the other hand, suppose $0\rightarrow G\otimes_RM\rightarrow G\otimes_RN$ is exact for any $G\in \mathscr{F}$. Considering the following commutative diagram of exact sequences:
	$$\xymatrix@R=20pt@C=50pt{
		0\ar[r]	&(F\otimes_SR)\otimes_RM\ar[r]^{}\ar[d]^{\cong} &(F\otimes_SR)\otimes_RN\ar[d]^{\cong} \\
		0\ar[r]&F\otimes_SM\ar[r]	& F\otimes_SN\\}$$
	we have  $M\rightarrow N$ is a pure monomorphism as $S$-modules.
	Since $S$ is a pure-semisimple ring, every pure exact sequence over $S$ splits. So $M\rightarrow N$ is an $(R,S)$-monomorphism. Thus $(\mathscr{A}_{(R,S)},\mathscr{E}_{(R,S)})$ is determined by $\mathscr{F}$.
	It follows from \cite[Theorem 6.6.4]{EJ11} that   $\mathscr{E}_{(R,S)}$ is enveloping. Hence every $R$-module has an $(R,S)$-injective envelope by Proposition \ref{equal-c}.
\end{proof}

\section{basic properties of $(R, S)_\star$-injective modules}
In this section, we introduce a new version of injective modules on a ring extension $f:S\rightarrow R$ by using purity over the ring $S$.

A short exact sequence $0\rightarrow  U\xrightarrow{f}
V\xrightarrow{g}  W\rightarrow  0$ of $R$-modules is called \emph{$(R, S)_\star$-exact} if it is pure  as  $S$-modules. In this case,  $f :U\rightarrow  V$ (resp., $g: V\rightarrow  W$)
is called an \emph{$(R, S)_\star$-monomorphism} (resp., \emph{$(R, S)_\star$-epimorphism}). An embedding map $i:M\hookrightarrow N$ is said to be \emph{$(R, S)_\star$-embedding}  if $i$ is an $(R, S)_\star$-monomorphism. In this case, $M$ is called an  \emph{$(R, S)_\star$-submodule} of $N$. And we say an $R$-homomorphism $a :M\rightarrow  N$ is an $(R, S)_\star$-homomorphism, denoted by $a\in\Hom_{(R,S)_\star}(M,N)$, if we have $i_a$ is an $(R, S)_\star$-monomorphism and $\pi_a$ is  an $(R, S)_\star$-epimorphism in the following natural factorization:

 $$\xymatrix@R=20pt@C=50pt{&\Im(a)\ar@{^{(}->}[rd]^{i_a}&\\
	M\ar@{->>}[ru]^{\pi_a}\ar[rr]_{a}& & N\\}$$

\begin{lemma}\label{rsmon} Suppose $a:M\rightarrow N$ and $b:N\rightarrow L$ are $R$-homomorphisms.
	\begin{enumerate}
		\item If $a$ and $b$ are $(R, S)_\star$-monomorphisms, then $ba$ is also an $(R, S)_\star$-monomorphism.
		\item If $a$ and $b$ are $(R, S)_\star$-epimorphisms, then $ba$ is also an $(R, S)_\star$-epimorphism.
		\item If $ba$ is an $(R, S)_\star$-monomorphism, then $a$ is also an $(R, S)_\star$-monomorphism.
		\item If $ba$ is an $(R, S)_\star$-epimorphism, then $b$ is also an $(R, S)_\star$-epimorphism.
		\item If $a$ is an $(R, S)_\star$-homomorphism and $b$ is a splitting epimorphism, then $ba$ is also an $(R, S)_\star$-homomorphism.
		\item If $a$ is a splitting monomorphism and $b$ is an  $(R, S)_\star$-homomorphism, then $ba$ is also an $(R, S)_\star$-homomorphism.	
	\end{enumerate}
\end{lemma}
\begin{proof} It is an easy exercise, and left to readers.
\end{proof}


\begin{definition}
Let  $f:S \rightarrow R$ be a ring extension. Then an $R$-module $M$ is said to be \emph{$(R, S)_\star$-injective}, if for every $(R, S)_\star$-exact sequence $0\rightarrow  U\xrightarrow{u}
V\xrightarrow{v}  W\rightarrow  0$
and every $R$-homomorphism $t :U\rightarrow M$, there is an $R$-homomorphism $s :V \rightarrow M$
such that $t=su$, that is, the following  commutative diagram with the bottom row $(R, S)_\star$-exact can be completed:
$$\xymatrix@R=30pt@C=50pt{	&& M&&\\
	0\ar[r]&U\ar[r]_{u}\ar[ru]^{t} &V\ar@{.>}[u]_{s}\ar[r]_v&W\ar[r]&0 \\
}$$
\end{definition}
Trivially, every injective $R$-module is $(R, S)_\star$-injective. And if $S$ is a von Neumann regular ring, then  $(R, S)_\star$-injective modules are exactly  injective $R$-modules. If $R=S$, then $(R, S)_\star$-injective modules are exactly  pure-injective $R$-modules.

The following  proposition shows that the class of all $(R, S)_\star$-injective modules is also closed under direct summands and direct products.

\begin{proposition}\label{dirsprod}
	Suppose $\{M_i\mid i\in\Gamma\}$ is a family of $R$-modules. Then  $\prod\limits_{i\in\Gamma}M_i$ is  an $(R, S)_\star$-injective module if and only if each $M_i$ is $(R, S)_\star$-injective.
\end{proposition}
\begin{proof}
Let $M=\prod\limits_{i\in\Gamma}M_i$, $a_i:M_i\rightarrow M$ and $\pi_i:M\rightarrow M_i$ be the natural $R$-homomorphisms.
	Let  $0\rightarrow  U\xrightarrow{u}
	V\xrightarrow{v}  W\rightarrow  0$ be an  $(R, S)_\star$-exact sequence.
	
	Suppose  $M_i$ is $(R, S)_\star$-injective for each $i\in\Gamma$. Let $t:U\rightarrow M$ be an $R$-homomorphism.  Since each $M_i$ is  $(R, S)_\star$-injective, then there is an $R$-homomorphism $s:V\rightarrow M$ such that $\pi_it=s_iu$:
	$$\xymatrix@R=30pt@C=50pt{&& M\ar@/^/[d]^{\pi_i}&&\\	
		&& M_i\ar@/^/[u]^{a_i}&&\\
		0\ar[r]&U\ar[r]_{u}\ar@/^1pc/[ruu]^{t} &V\ar[u]^{s_i}\ar@/_3pc/@{.>}[uu]_{s}\ar[r]_v&W\ar[r]&0 \\
	}$$	
	Set $s=\prod\limits_{i\in\Gamma} s_i$. Then $t=su$. So $M$ is  $(R, S)_\star$-injective.
	
	Now suppose  $M$ is $(R, S)_\star$-injective. Let $i\in\Gamma$ and $w:U\rightarrow M_i$ be an $R$-homomorphism. 	
	Then there is an $R$-homomorphism $s:V\rightarrow M$ such that $a_iw=su$:
	$$\xymatrix@R=30pt@C=50pt{&& M\ar@/^/[d]^{\pi_i}&&\\	
		&& M_i\ar@/^/[u]^{a_i}&&\\
		0\ar[r]&U\ar[r]_{u}\ar[ru]^{w} &V\ar@{.>}[u]^t\ar@/_3pc/@{>}[uu]_{s}\ar[r]_v&W\ar[r]&0 \\
	}$$	
	Set $t=\pi_is$.	Then $w=tu$. So $M_i$ is  $(R, S)_\star$-injective.
\end{proof}

Let $M$ be an $R$-module,  $\epsilon:M\rightarrow  \Hom_S(R,M)$ and $\delta:M\rightarrow M^{++}$ be the natural $R$-homomorphisms. Then it is easy to very that  $\Hom_S(R,\delta): \Hom_S(R,M)\rightarrow\Hom_S(R,M^{++})$ is also an $(R, S)_\star$-monomorphism.	We always set the composition $$\varepsilon:=\Hom_S(R,\delta)\epsilon:M\rightarrow \Hom_S(R,M^{++}).$$ Then $\varepsilon$ is an  $(R, S)_\star$-monomorphism by Lemma \ref{rsmon}(1).

Let $f:S\rightarrow R$ be a ring extension and $M$ be an $R$-module. We always denote by $$M^{\ddag}:=\Hom_S(R,M^{++})$$ in the rest of this paper.
\begin{proposition}\label{ddag}
	Let $M$ be an $R$-module. Then $M^{\ddag}$ is an $(R, S)_\star$-injective module.
\end{proposition}
\begin{proof}
	Let  $0\rightarrow  U\xrightarrow{u}
	V\xrightarrow{v}  W\rightarrow  0$ be an  $(R, S)_\star$-exact sequence. Note that $M^{++}$ is also a pure-injective $S$-module. So $\Hom_S(u,M^{++}):\Hom_S(V,M^{++})\rightarrow \Hom_S(U,M^{++})$ is an $S$-epimorphism. Note that there is a natural $R$-isomorphism $$\Hom_R(N,\Hom_S(R,M^{++}))\cong\Hom_S(N,M^{++}).$$
	Using this  with $M=U$ and $M=V$, we have an 	$R$-epimorphism $$\Hom_{R}(u,M^{\ddag}):\Hom_{R}(V,M^{\ddag})\rightarrow \Hom_{R}(U,M^{\ddag})$$
	which means  $M^{\ddag}$ is an $(R, S)_\star$-injective module.
\end{proof}

\begin{theorem}\label{rsinj} Let $M$ is an $R$-module. Then the following statements are equivalent:
	\begin{enumerate}
		\item $M$ is an $(R, S)_\star$-injective module;
		\item every short $(R, S)_\star$-exact sequence starting at $M$ splits;
		\item  if for every $(R, S)_\star$-exact sequence $0\rightarrow  U\xrightarrow{u}
		V\xrightarrow{v}  W\rightarrow  0$
		and every $(R,S)_\star$-homomorphism $t :U\rightarrow M$, there is an $(R,S)_\star$-homomorphism $s :V \rightarrow M$
		such that $t=su$;
		\item the natural $R$-homomorphism $\varepsilon:M\rightarrow M^{\ddag}$ splits;	
		\item $M$ is a direct summand of $N^{\ddag}$ for some $R$-module $N$.
	\end{enumerate}
\end{theorem}

\begin{proof} $(1)\Rightarrow (2)$ and $(3)\Rightarrow (2)$: Follow by the completement of the following commutative of commutative diagram:
	$$\xymatrix@R=30pt@C=50pt{	&& M&&\\
		0\ar[r]&M\ar[r]_{m}\ar@{=}[ru]^{\Id_M} &V\ar@{.>}[u]_{s}\ar[r]_v&W\ar[r]&0 \\
	}$$
	
 $(2)\Rightarrow (1)$: 	Let  $0\rightarrow  U\xrightarrow{u}
 V\xrightarrow{v}  W\rightarrow  0$ be an  $(R, S)_\star$-exact sequence and $t:U\rightarrow M$ be an $R$-homomorphism. Then is is a pullback-pushout diagram:
 $$\xymatrix@R=30pt@C=50pt{	0\ar[r]&M\ar[r]^m& X\ar[r]^x&W\ar[r]&0\\
 	0\ar[r]&U\ar[r]_{u}\ar[u]^{t} &V\ar[u]_{s}\ar[r]_v&W\ar@{=}[u]\ar[r]&0 \\
 }$$
 Note that the top row is also $(R, S)_\star$-exact. So it splits. Let $m':X\rightarrow M$ be an $R$-homomorphism such that $m'm=\Id_M$. Then $m's$ is the required homomorphism.

 $(2)\Rightarrow (3)$: 	 Similar to  $(2)\Rightarrow (1)$. We only note Lemma \ref{rsmon}(5) holds and if $t$ is an $(R, S)_\star$-homomorphism, then so is $s$.

 $(5)\Rightarrow (1)$:  Follows by Proposition \ref{dirsprod} and Proposition \ref{ddag}.

  $(1)\Rightarrow (4)$: Since the natural $R$-homomorphism $\varepsilon:M\rightarrow M^{\ddag}$ is an $(R, S)_\star$-monomorphism, it splits.

   $(4)\Rightarrow (5)$: Trivial.
\end{proof}

\section{$(R, S)_\star$-injective envelopes}

Let  $f:S \rightarrow R$ be a ring extension. We call an $(R, S)_\star$-monomorphism  $\sigma:L\rightarrow M$ of $R$-modules  \emph{$(R, S)_\star$-essential} provided that every  $(R, S)_\star$-homomorphism $\tau:M\rightarrow N$ is an $(R, S)_\star$-monomorphism, whenever $\tau\sigma$ is an $(R, S)_\star$-monomorphism.

The following result shows that to check an $(R, S)_\star$-monomorphism is $(R, S)_\star$-essential, we only need to consider all $(R,S)_\star$-epimorphisms in the above definition.
\begin{lemma}\label{ess}
	An $(R, S)_\star$-monomorphism $\sigma:L\hookrightarrow M$ is  $(R, S)_\star$-essential if and only if every  $(R, S)_\star$-epimorphism $\tau:M\rightarrow N$ is an $R$-isomorphism, whenever $\tau\sigma$ is an $(R, S)_\star$-monomorphism.
\end{lemma}
\begin{proof}
	Let $\tau:M\rightarrow K$ be an $(R, S)_\star$-homomorphism such that $L\stackrel{ \sigma}{\hookrightarrow}  M\xrightarrow{\tau} K$ is an  $(R, S)_\star$-monomorphism, i.e., the natural composition $$\delta:L\stackrel{ \sigma}{\hookrightarrow}  M\stackrel{ \pi_\tau}{\twoheadrightarrow} \Im(\tau)\stackrel{ i_\tau}{\hookrightarrow} K$$ is an $(R, S)_\star$-monomorphism,
	It follows by Lemma \ref{rsmon}(3) that  $L\stackrel{ \sigma}{\hookrightarrow}  M\stackrel{ \pi_\tau}{\twoheadrightarrow} \Im(\tau)$ is also an $(R, S)_\star$-monomorphism.
	By assumption, $\pi_\tau$ is an $R$-isomorphism. Consequently,
	$\tau=i_\tau\pi_\tau$ is an $(R, S)_\star$-monomorphism. Hence
	$\sigma$ is  $(R, S)_\star$-essential.	
\end{proof}

Let $f:S\rightarrow R$ be a ring extension, $M$ be an $R$-module and $\mathscr{C}$ be a  class  of $R$-modules. We call an $(R,S)_\star$-homomorphism $f:M\rightarrow C$ with  $C\in \mathscr{C}$ is said to be an  \emph{$(R,S)_\star$-$\mathscr{C}$-preenvelope}  provided that  for any $(R,S)_\star$-homomorphism $f':M\rightarrow C'$ with $C'\in \mathscr{C}$, there is an $(R,S)_\star$-homomorphism $h:C\rightarrow C'$ such that the following diagram is commutative:$$\xymatrix@R=20pt@C=50pt{M\ar[r]^{f}\ar[rd]_{f'} &C\ar@{.>}[d]^{h} \\
	& C'\\}$$
If, moreover, every $(R,S)_\star$-endomorphism $h$ such that $f=hf$  is an automorphism ($f$ is said to be left $(R,S)_\star$-minimal), that is, the commutativity of
$$\xymatrix@R=20pt@C=50pt{M\ar[r]^{f}\ar[rd]_{f} &C\ar@{.>}[d]^{h} \\
	& C\\}$$ implies $h$ is an automorphism,
then $f$ is said to be an
\emph{$(R,S)_\star$-$\mathscr{C}$-envelope}.

\begin{theorem}\label{equal}
	Let $\mathscr{E}_{(R,S)_\star}$ be the  class  of all $(R, S)_\star$-injective modules. Suppose $\sigma:M\rightarrow E$ is an $R$-homomorphism with $E\in \mathscr{E}_{(R,S)_\star}$. Then $\sigma$ is an  $(R,S)_\star$-$\mathscr{E}_{(R,S)_\star}$-envelope if and only if it is an $(R, S)_\star$-essential $(R, S)_\star$-monomorphism.
\end{theorem}
\begin{proof} Let $\sigma:M\rightarrow E$ be an $(R,S)_\star$-$\mathscr{E}_{(R,S)_\star}$-envelope of $M$. Then $\sigma$ is an $(R, S)_\star$-monomorphism. Indeed, let $\varepsilon:M\rightarrow M^{\ddag}$ be the natural $R$-homomorphism. Since $M^{\ddag}$ is  $(R, S)_\star$-injective, then there is an $(R,S)_\star$-homomorphism $h:E\rightarrow M^{\ddag}$ such that the following diagram  commutative
	$$\xymatrix@R=20pt@C=50pt{M\ar[r]^{\sigma}\ar@{_{(}->}[rd]_{\varepsilon} &E\ar@{.>}[d]^{h} \\
		& M^{\ddag}\\}$$
	Since $\varepsilon$ is an $(R, S)_\star$-monomorphism, so is  $\sigma$ by Lemma \ref{rsmon}(3).
	We claim that $\sigma$ is $(R, S)_\star$-essential. Indeed, let $\tau:E\rightarrow N$ be an $(R, S)_\star$-epimorphism such that $\tau\sigma$ is an $(R, S)_\star$-monomorphism. Then there is an $(R,S)_\star$-homomorphism $\lambda: N\rightarrow E$ as $E$ is $(R, S)_\star$-injective.
	
	$$\xymatrix@R=20pt@C=50pt{
		&E\ar@{->>}[d]^{\tau}\\
		M\ar[r]^{\tau\sigma}\ar[ru]^{\sigma}\ar[rd]_{\sigma} &N\ar[d]^{\lambda} \\
		& E\\}$$
	Since $\sigma$ is left $(R,S)_\star$-minimal, $\lambda\tau$ is an isomorphism. So $\tau$ is an $(R, S)_\star$-monomorphism by Lemma \ref{rsmon}(3), and hence $\tau$ is  an isomorphism. Consequently,  $\sigma$ is $(R, S)_\star$-essential by Lemma \ref{ess}.
	
	On the other hand, suppose $\sigma:M\hookrightarrow E$ is an $(R, S)_\star$-essential $(R, S)_\star$-monomorphism with $E$ an  $(R, S)_\star$-injective module. Then by Theorem \ref{rsinj} for any   $(R,S)_\star$-homomorphism $f':M\rightarrow E'$ with $E'\in \mathscr{E}_{(R,S)_\star}$, there is an $(R,S)_\star$-homomorphism $h:E\rightarrow E'$ such that the following diagram is commutative:$$\xymatrix@R=20pt@C=50pt{M\ar@{^{(}->}[r]^{\sigma}\ar[rd]_{f'} &E\ar@{.>}[d]^{h} \\
		& E'\\}$$
	We claim that $\sigma$ is left $(R,S)_\star$-minimal. Indeed, let $\tau:E\rightarrow E$ be an  $(R,S)_\star$-homomorphism such that $\tau\sigma=\sigma$. Since $\sigma$ is $(R,S)_\star$-essential, $\tau$ is an $(R,S)_\star$-monomorphism. Consider the $(R,S)_\star$-exact sequence $0\rightarrow E\xrightarrow{\tau} E\rightarrow C\rightarrow0$. Since $E$ is $(R,S)_\star$-injective, there is an $(R,S)_\star$-homomorphism $\tau':E\rightarrow E$ such that $\tau'\tau=\Id_E$ by Theorem \ref{rsinj}. So $E=\Im(\tau)\oplus\Ker(\tau')$ as $R$-modules. Note that $\Im(\sigma)\subseteq\Im(\tau)$ and we have the following commutative diagram:
	$$\xymatrix@R=20pt@C=50pt{&\Im(\tau)\oplus\Ker(\tau')\ar@{.>}[dd]^{\delta:=\begin{pmatrix}
				1 & 0\\
				0 & 0\\
		\end{pmatrix}}\\	
		M\ar@{^{(}->}[ru]^{\sigma}\ar@{_{(}->}[rd]_{\sigma} &
		\\
		&\Im(\tau)\oplus\Ker(\tau')\\}$$
	Since $\sigma$ is $(R,S)_\star$-essential, $\delta$ is a monomorphism and so $\Ker(\tau')=0$. Hence $E=\Im(\tau)$ which implies that $\tau$ is also an epimorphism, and so is an isomorphism.  Therefore $\tau$ is left $(R,S)_\star$-minimal.   Consequently, $\sigma$ is an $\mathscr{E}_{(R,S)_\star}$-envelope of $M$.	
\end{proof}

Let $\mathscr{E}_{(R,S)_\star}$ be the  class  of all $(R, S)_\star$-injective modules. Then we always call   $(R,S)_\star$-$\mathscr{E}_{(R,S)_\star}$-envelopes $(R, S)_\star$-injective  $(R,S)_\star$-envelopes.
The following main theorem of this section shows that  $(R, S)_\star$-injective $(R, S)_\star$-envelopes always exist.
\begin{theorem}\label{main} Let $f:S\rightarrow R$ be a ring extension. Then every $R$-module has an $(R, S)_\star$-injective  $(R,S)_\star$-envelope.
\end{theorem}
\begin{proof}  Let $M$ be an $R$-module. Then $M^{\ddag}$ is an $(R, S)_\star$-injective module by Proposition \ref{ddag}. Note that natural map $\varepsilon:M\rightarrow M^{\ddag}$ is an  $(R, S)_\star$-monomorphism. So we can see $M$ as an $R$-submodule of $ M^{\ddag}$ by $\varepsilon$. Set $$\Lambda=\{M\stackrel{ \sigma}{\hookrightarrow} N\mid N\leq M^{\ddag},\ \sigma\  \mbox{is an}\  (R,S)_\star\mbox{-essential}\ (R,S)_\star\mbox{-embedding map}\}.$$	
	Since $\Id_M\in \Lambda$, 	$\Lambda$ is not empty. We order $\Lambda$ as follows: let $\sigma_i:M\rightarrow N_i$ and  $\sigma_j:M\rightarrow N_j$ in $\Lambda$. Define $\sigma_i\leq \sigma_j$ if and only if there is an $(R, S)_\star$-embedding map $\delta_{ij}$ such that the following diagram is  commutative:	$$\xymatrix@R=20pt@C=50pt{M\ar@{^{(}->}[r]^{\sigma_i}\ar@{_{(}->}[rd]_{\sigma_j} &N_i\ar@{^{(}->}[d]^{\delta_{ij}} \\
		& N_j\\}$$
	
	\textbf{Claim 1: $\Lambda$ has a maximal element} $\boldsymbol{\sigma:M\hookrightarrow N}$\textbf{.} Indeed, let $\Gamma=\{\sigma_i\mid i\in A\}$ be a total subset of 	$\Lambda$. Then $\{N_i,\delta_{ij}\mid i\leq j\in A\}$ is a directed system of $R$-modules. Set $N=\lim\limits_{\longrightarrow A}N_i=\bigcup\limits_{i\in A}N_i$ and $\sigma=\lim\limits_{\longrightarrow A}\sigma_i:M\rightarrow N$. 
	It follows by \cite[Corollary 2.21(1)]{GT12} that $\sigma$ is also an $(R, S)_\star$-embedding map. We will show that $\sigma$ is $(R, S)_\star$-essential. Let $\tau:N\rightarrow L$ be an $(R,S)_\star$-homomorphism such that  composition map $\phi:M\stackrel{ \sigma}{\hookrightarrow} N\xrightarrow{\tau} L$ is an $(R, S)_\star$-monomorphism.  For each $i,j\in A$, consider the following commutative diagram:
	
	$$
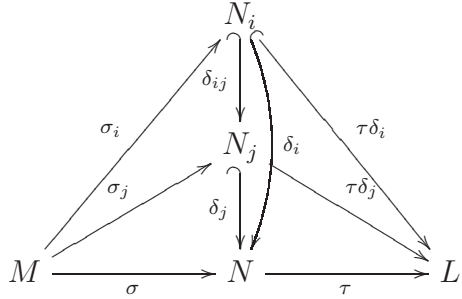
\xymatrix@R=30pt@C=60pt{&N_i\ar@{_{(}->}[d]_{\delta_{ij}}\ar[rdd]^{\tau\delta_i}\ar@{^{(}->}@/^1pc/[dd]^{\delta_{i}}&\\ &N_j\ar@{_{(}->}[d]_{\delta_j}\ar[rd]^{\tau\delta_j}&\\
		M\ar[ruu]^{\sigma_i}\ar[ru]^{\sigma_j}\ar[r]_{\sigma}&N\ar[r]_{\tau}	& L\\}$$
	Since each $\sigma_{i}$ is $(R, S)_\star$-essential and $\tau\delta_{i}\sigma_{i}=\tau\sigma$ is an $(R, S)_\star$-monomorphism, we have $\tau\delta_{i}$ is also an $(R, S)_\star$-monomorphism.  Hence $\tau=\lim\limits_{\longrightarrow A} \tau\delta_{i}$ is also an $(R, S)_\star$-monomorphism  by \cite[Corollary 2.21(1)]{GT12} again. Consequently,  $\sigma$ is $(R, S)_\star$-essential. It follows that  $\sigma$ is an upper bound of $\Gamma$ in $\Lambda$.
	It follows by Zorn Lemma that $\Lambda$ has a maximal element. We also it denoted by $\sigma:M\hookrightarrow N.$
	
	\textbf{Claim 2: $\boldsymbol{N}$ is an $\boldsymbol{(R,S)_\star}$-injective module.} First, we will show $N$ has no nontrivial $(R, S)_\star$-essential $(R, S)_\star$-embedding map. Indeed, let $\mu:N\hookrightarrow X$ be an $(R, S)_\star$-essential $(R, S)_\star$-monomorphism. Since $ M^{\ddag}$ is $(R, S)_\star$-injective, there is an $(R,S)_\star$-homomorphism $\xi:X\rightarrow  M^{\ddag}$ such that the following diagram is commutative:
	$$
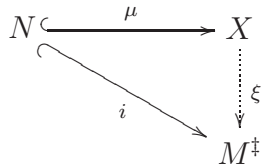
\xymatrix@R=30pt@C=60pt{N\ar@{^{(}->}[r]^{\mu}\ar@{_{(}->}[rd]_{i} &X\ar@{.>}[d]^{\xi} \\
		& M^{\ddag}\\}$$
	Since $\mu$ is  $(R, S)_\star$-essential  and $i=\xi\mu$ is an $(R, S)_\star$-monomorphism, we have $\xi$ is an $(R, S)_\star$-monomorphism. Let $\beta:M\stackrel{ \sigma}{\hookrightarrow}N\stackrel{ \mu}{\hookrightarrow}X\xrightarrow[\cong]{\pi_\xi}\xi(X)$ be the composition map. Then $\beta\in\Gamma$ and $\beta\geq \sigma$. Since $\sigma$ is maximal in $\Gamma$, we have $N=X$. Consequently, $N$ has no nontrivial $(R, S)_\star$-essential $(R, S)_\star$-embedding map. Now, let $\Psi$ the set of $(R, S)_\star$-submodules $H\subseteq  M^{\ddag}$ with $H\cap N=0$. Since $0\in \Psi$, $\Psi$ is nonempty.  Order $\Psi$ as follows $H_i\leq H_j$ if and only if $H_i$ is an $(R, S)_\star$-submodule of  $H_j$. Next, we will show $\Psi$ has a maximal element. Let $\Phi=\{H_i\mid i\in B\}$ be a total ordered subset of $\Psi$. Consider the following commutative diagram:
	$$\xymatrix@R=30pt@C=100pt{H_i\ar@{^{(}->}[d]^{\delta_{ij}}\ar[rdd]^{\alpha_i}\ar@{_{(}->}@/_1pc/[dd]_{\delta_{i}}&\\ H_j\ar@{^{(}->}[d]^{\delta_j}\ar[rd]^{\alpha_j}&\\
		\lim\limits_{\longrightarrow B}H_i\ar[r]^{\alpha}	&  M^{\ddag}\\}$$
	Since each $\alpha_i$ is an $(R, S)_\star$-monomorphism, $\alpha=\lim\limits_{\longrightarrow B}\alpha_i$ is an $(R, S)_\star$-monomorphism by \cite[Corollary 2.21(1)]{GT12}. It is easy to verify $ H'\cap N=0$. Hence,  $ H'\in\Psi$ is an upper bound of $\Phi$. It follows by Zorn Lemma that  $\Psi$ has a maximal element, which is also denoted by $H'$. Then $H'$ is an $(R, S)_\star$-submodule of $ M^{\ddag}$ with $H'\cap N=0$.  So the canonical map $\pi: M^{\ddag}\twoheadrightarrow  M^{\ddag}/H'$ induces  an $(R, S)_\star$-monomorphism
	$\pi|_N:N\hookrightarrow M^{\ddag}/H'$. Then, we will show $\pi|_N$ is $(R, S)_\star$-essential.
	Let $\nu: M^{\ddag}/H'\twoheadrightarrow  M^{\ddag}/H$ be a  natural $(R, S)_\star$-epimorphism such that $\nu\pi_N$ is an $(R, S)_\star$-monomorphism, that is, the following diagram is commutative:
	$$\xymatrix@R=30pt@C=60pt{N\ar@{^{(}->}[r]^{\pi|_N}\ar@{_{(}->}[rd]_{\nu\pi|_N} & M^{\ddag}/H'\ar@{->>}[d]^{\nu} \\
		& M^{\ddag}/H\\}$$
	Then $H\cap N=0,$ and $H$, which contains $H'$, is an $(R, S)_\star$-submodule of $M^{\ddag}$ by \cite[Lemma 2.25(b)]{GT12}. By the maximality of $H'$, we have $H=H'$, and hence $\nu$ is an identity.
	Consequently, $\pi|_N$ is $(R, S)_\star$-essential by Lemma \ref{ess}. Finally, we will show $N$ is $(R, S)_\star$-injective.  Since $N$  has no nontrivial $(R, S)_\star$-essential $(R, S)_\star$-monomorphism as above,
	we have $\pi|_N$ is actually an isomorphism as $R$-modules. Thus the exact sequence
	$$0\rightarrow H'
	\rightarrow
	M^{\ddag}
	\xrightarrow{(\pi|_N)^{-1}\pi} N\rightarrow 0$$
	splits as $R$-modules. It follows that $N$ is an direct summand of $ M^{\ddag}$ as $R$-modules. Consequently, $N$ is an $(R, S)_\star$-injective module by Theorem \ref{rsinj}.
	
	In conclusion,  it follows by Theorem \ref{equal} that $\sigma:M\rightarrow N$ is an $(R, S)_\star$-injective  $(R,S)_\star$-envelope.
\end{proof}

Similar to the proof of results in Section 2, we have the following result.
\begin{theorem}\label{rsinjenve}
	Let  $f:S\rightarrow R$ be a ring extension.	Let $\mathscr{E}_{(R,S)_\star}$ be  the  class  of all $(R,S)_\star$-injective modules and $\mathscr{A}_{(R,S)_\star}$  the  class  of all $(R,S)_\star$-monomorphisms.
	Then $(\mathscr{A}_{(R,S)_\star},\mathscr{E}_{(R,S)_\star})$ is an  injective structure determined by $$\mathscr{F}=\{F\otimes_SR\mid F\ \mbox{is a finitely presented right}\ S\mbox{-module}\}.$$ Consequently, every $R$-module has an $\mathscr{E}_{(R,S)_\star}$-envelope.
\end{theorem}
\begin{proof}
	The proof is similar to these of Proposition \ref{injstru} and Theorem \ref{psprsijn}, and so we omit it.
\end{proof}

Using the above theorem, we find the $(R,S)_\star$-injective envelope of an $R$-module coincides with its $(R,S)_\star$-injective $(R,S)_\star$-envelope.
\begin{corollary}
	Let $\mathscr{E}_{(R,S)_\star}$ be  the  class  of all $(R,S)_\star$-injective modules, and $M$ be an $R$-module. Then	 $\sigma:M\rightarrow E$ is an  $\mathscr{E}_{(R,S)_\star}$-envelope if and only if  it is an $(R,S)_\star$-injective $(R,S)_\star$-envelope.
\end{corollary}
\begin{proof} Let $\sigma':M\rightarrow E'$ be an  $\mathscr{E}_{(R,S)_\star}$-envelope. It follows by the commutative diagram	$$\xymatrix@R=20pt@C=50pt{M\ar[r]^{\sigma'}\ar@{_{(}->}[rd]_{\varepsilon} &E'\ar@{.>}[d]^{h} \\
		& M^{\ddag}\\}$$
that 	$\sigma'$ is an $(R,S)_\star$-monomorphism (see Lemma \ref{rsmon}(3)). Hence, $\sigma'$ is an  $(R,S)_\star$-injective $(R,S)_\star$-envelope.

 On the other hand, let $\sigma:M\rightarrow E$ be an  $(R,S)_\star$-injective $(R,S)_\star$-envelope. It follows by Theorem \ref{rsinjenve} that there is an $\mathscr{E}_{(R,S)_\star}$-envelope $\sigma':M\rightarrow E'$, which is an $(R,S)_\star$-monomorphism. So there are  $(R,S)_\star$-homomorphisms $\tau:E'\rightarrow E$ and $\lambda:E\rightarrow E'$ such that the following diagram is commutative :	$$\xymatrix@R=20pt@C=50pt{
	&E'\ar[d]^{\tau}\\
	M\ar[r]^{\sigma}\ar[ru]^{\sigma'}\ar[rd]_{\sigma'} &E\ar[d]^{\lambda} \\
	& E'\\}$$
Then $\lambda\tau$ is an isomorphism, and so  $\lambda$ is a splitting epimorphism.  Then considering the following commutative diagram:	$$\xymatrix@R=20pt@C=50pt{
	&E\ar[d]^{\lambda}\\
	M\ar[r]^{\sigma'}\ar[ru]^{\sigma}\ar[rd]_{\sigma} &E'\ar[d]^{\delta} \\
	& E\\}$$
we have $\delta$ is  an $(R,S)_\star$-homomorphism by Theorem \ref{rsinj}. So $\delta\lambda$ is also an $(R,S)_\star$-homomorphism. Thus $\delta\lambda$ is an isomorphism. Hence, $\lambda$ is  a splitting monomorphism, and so is an isomorphism. Consequently, $\sigma:M\rightarrow E$ is an  $\mathscr{E}_{(R,S)_\star}$-envelope.
\end{proof}

\end{document}